\documentclass[11pt,twoside]{article}
\usepackage{etex}
\usepackage[dvipsnames]{xcolor} 
\usepackage{color}
\usepackage{booktabs}
\usepackage{amsthm}
\usepackage{amsfonts,amssymb,amsmath,amsxtra,url,float} 
\allowdisplaybreaks[4] 
\usepackage[colorlinks,
  linkcolor=magenta, %
  anchorcolor=Periwinkle,
  citecolor=violet,
  urlcolor=blue
  ]{hyperref} 
\usepackage{enumitem}
\usepackage{geometry} 
\usepackage{rotating} 
\usepackage{lscape} 
\usepackage{yhmath} 
\usepackage{multirow}
\usepackage{graphicx} 
\usepackage{subfigure} 
\usepackage{tikz}
\usepackage{pgfplots}
\usepackage{tikz-3dplot}
\usetikzlibrary{patterns}
\usetikzlibrary{3d,calc}
\usetikzlibrary{decorations.pathreplacing,decorations.markings}
 \tikzset{
  on each segment/.style={
    decorate,
    decoration={
      show path construction,
      moveto code={},
      lineto code={
        \path [#1]
        (\tikzinputsegmentfirst) -- (\tikzinputsegmentlast);
      },
      curveto code={
        \path [#1] (\tikzinputsegmentfirst)
        .. controls
        (\tikzinputsegmentsupporta) and (\tikzinputsegmentsupportb)
        ..
        (\tikzinputsegmentlast);
      },
      closepath code={
        \path [#1]
        (\tikzinputsegmentfirst) -- (\tikzinputsegmentlast);
      },
    },
  },
  mid arrow/.style={postaction={decorate,decoration={
        markings,
        mark=at position 0.6 with {\arrow[#1]{stealth}} 
      }}},
}
\usetikzlibrary{arrows}
\usetikzlibrary{trees}

\usetikzlibrary{matrix}
\usetikzlibrary{patterns}
\usetikzlibrary{shadings} 
\usepackage{fancyhdr} 
\def\headertitle{Homological Bounds of Gentle algebras}
\def\fstpage{1} 
\def\page{$\begin{matrix} {\color{white}0} \\ \thepage \end{matrix}$} 

\usepackage[all]{xy} 
\usepackage{dsfont} 
\usepackage{cite}
\usepackage{mathrsfs} 
\numberwithin{figure}{section}
\usepackage{marginnote} 
\usepackage{graphicx} 
\usepackage{multicol} 

\usepackage{enumitem}
\setenumerate[1]{itemsep=0pt,partopsep=0pt,parsep=\parskip,topsep=3pt}
\setitemize[1]{itemsep=0pt,partopsep=0pt,parsep=\parskip,topsep=3pt}
\setdescription{itemsep=0pt,partopsep=0pt,parsep=\parskip,topsep=3pt}
\setlist[itemize]{leftmargin=35pt}
\setlist[enumerate]{leftmargin=35pt}
\usepackage{changepage} 

\newtheorem{theorem}{Theorem}[section]
\newtheorem{lemma}[theorem]{Lemma}
\newtheorem{corollary}[theorem]{Corollary}
\newtheorem{main theorem}[theorem]{Main Theorem}
\newtheorem{proposition}[theorem]{Proposition}
\newtheorem{definition}[theorem]{Definition}

\newtheorem{remark}[theorem]{Remark}
\newtheorem{example}[theorem]{Example}

\newtheorem{question}[theorem]{Question}

\usetikzlibrary{arrows}

\numberwithin{equation}{section}

\def\orcid{
\begin{tikzpicture}[baseline=-1mm]
\filldraw[Green!35] (0,0) circle (5pt);
\filldraw[white] (0,0) node{\tiny\textbf{iD}};
\end{tikzpicture}
}
\def\orcid{
\begin{tikzpicture}[baseline=-1mm]
\filldraw[Green!35] (0,0) circle (5pt);
\filldraw[white] (0,0) node{\tiny\textbf{iD}};
\end{tikzpicture}
}
\newcommand{\ORCID}[1]{ORCID: \href{https://orcid.org/#1}{#1}}
\newcommand{\ORCIDNOTATION}[1]{\href{https://orcid.org/#1}{\orcid}}
\def\EnglishTitle{Homological Bounds of Gentle algebras}
\def\EnglishFundings{
Yu-Zhe Liu is supported by
the National Natural Science Foundation of China (Grant Nos. 12561008, 12401042, and 12171207);
Guizhou Provincial Basic Research Program (Natural Science) (Grant Nos. ZD[2025]085 and ZK[2024]YiBan066);
and Scientific Research Foundation of Guizhou University (Grant Nos. [2022]53, [2022]65, [2023]16). \\
{\color{white}......}
Xin Ma is supported by
Central Plains Science and Technology Innovation Youth Top-notch Talent and Henan University of Engineering (DKJ2019010). \\
{\color{white}......}
Chao Zhang is supported by
the National Natural Science Foundation of China (Grant No. 12461006);
and Guizhou Provincial Basic Research Program (Natural Science) (Grant No. ZD[2025]085).
}
\def\FirstAuthorORICD{
0009-0005-1110-386X}
\def\SecondAuthorORICD{
0009-0007-5987-9848}
\def\ThirdAuthorORICD{
0009-0007-7780-0160}

\def\PaperAuthorsENname{
Yu-Zhe Liu
$^{\ref{Author1}, \ORCIDNOTATION{\FirstAuthorORICD}\ref{orcid1}}$,
Xin Ma
$^{\ref{Author2}, \ORCIDNOTATION{\SecondAuthorORICD}\ref{orcid2},~\ref{CorrespondingAuthor}}$,
Jiacheng Xu
$^{\ref{Author1}, \ORCIDNOTATION{\ThirdAuthorORICD}\ref{orcid3}}$,
Chao Zhang
$^{\ref{Author1}, ~\ref{CorrespondingAuthor}}$
}
\def\FirstEnOrgani{School of Mathematics and Statistics, Guizhou University, Guiyang 550025, Guizhou, China}
\def\SecondEnOrgani{College of Science, Henan University of Engineering, Zhengzhou 451191, Henan, China}

\def\FirstEmail{\url{yzliu3@163.com}/\url{liuyz@gzu.edu.cn} (Y.-Z. Liu); \\ \url{xjcgzu823@163.com}/\url{3191970976@qq.com} (J. Xu); \\ \url{zhangc@amss.ac.cn} (C. Zhang)}
\def\SecondEmail{\url{maxin@haue.edu.cn}}

\def\NN{\mathbb{N}} 
\def\ZZ{\mathbb{Z}} 

\newcommand{\Ima}{\operatorname{Im}}
\newcommand{\Pic}{Figure\ }
\newcommand{\modcat}{\mathsf{mod}}

\newcommand{\ind}{\mathsf{ind}}

\def\kk{\Bbbk} 
\def\Q{\mathcal{Q}} 
\def\I{\mathcal{I}}
\def\top{\mathrm{top}}
\def\soc{\mathrm{soc}}
\def\HH{\mathrm{H}}

\newcommand{\e}{\varepsilon}
\newcommand{\Ext}{\mathrm{Ext}} %
\newcommand{\Tor}{\mathrm{Tor}} %

\newcommand{\w}[1]{\widetilde{#1}}
\newcommand{\To}[1]{\mathop{-\!\!\!-\!\!\!\longrightarrow}\limits^{#1}}

\def\alg{\mathit{\Lambda}}

\def\heart{{\color{red}\pmb{\heartsuit}}}

\def\compos{\ \lower-0.2ex\hbox{\tikz\draw (0pt, 0pt) circle (.1em);} \ }

\newcommand{\defines}{\it\color{red}}

\title{\bf \EnglishTitle$^{\color{red}\dag}$
\footnotetext[2]{ \tiny \EnglishFundings}
}

\vspace{5mm}

\author{\PaperAuthorsENname}
\date{ }

\begin{document}



\thispagestyle{empty}

\maketitle

\begin{enumerate}[label=\textbf{\color{red}$\ddag$}]
  \item \footnotesize
    \begin{center}
      Corresponding author
    \end{center} \label{CorrespondingAuthor}
\end{enumerate}


\begin{enumerate}[leftmargin=6.8cm] \footnotesize
  \item[\orcid]
      \ORCID{\FirstAuthorORICD}
      \label{orcid1} 
  \item[\orcid]
      \ORCID{\SecondAuthorORICD}
      \label{orcid2} 
  \item[\orcid]
      \ORCID{\ThirdAuthorORICD}
      \label{orcid3} 
\end{enumerate}

\vspace{2mm}
\begin{enumerate}[label=\textbf{\color{red}\arabic*}] \footnotesize
  \item
    \begin{center}
      \FirstEnOrgani

      E-mail: \FirstEmail
    \end{center}
    \label{Author1}

  \item
    \begin{center}
      \SecondEnOrgani

      E-mail: \SecondEmail
    \end{center}
    \label{Author2}

%
%
%
%
\end{enumerate}




\vspace{1mm}


\begin{adjustwidth}{1cm}{1cm}
  \noindent \footnotesize
  \textbf{Abstract}:
This paper studies the homological bounds of gentle algebras, i.e., the upper bounds for the sum of the projective and injective dimensions of indecomposable modules over gentle algebras. We provide conditions under which this sum is strictly less than twice the global dimension, and as an application, we give a characterization of quasi-tilted gentle algebras.
\vspace{1mm}

     \label{2020MSC}

\vspace{1mm}

     \label{Keywords}
\end{adjustwidth}



\def\la{\langle} 
\def\ra{\rangle} 
\def\lala{\langle\!\langle}
\def\rara{\rangle\!\rangle}
\def\=<{\leqslant}
\def\>={\geqslant}
\def\s{\mathfrak{s}}
\def\t{\mathfrak{t}}
\def\C{\mathscr{C}}
\def\str{\mathbf{s}}
\def\band{\mathbf{b}}
\def\Str{\mathsf{Str}}
\def\Band{\mathsf{Ban}}
\def\M{\mathds{M}}
\def\cls{\mathrm{cls}}

\def\pdim{\mathrm{proj.dim}}
\def\idim{\mathrm{inj.dim}}
\def\gldim{\mathrm{gl.dim}}
\def\fdim{\mathrm{f.dim}}
\def\Hbound{\mathrm{hb.dim}}

\def\rmI{\mathrm{I}}
\def\rmII{\mathrm{II}}
\def\rmIII{\mathrm{III}}
\def\rmIV{\mathrm{IV}}
\def\Left{\mathrm{L}}
\def\Right{\mathrm{R}}

\def\bsm{\begin{smallmatrix}}
\def\esm{\end{smallmatrix}}

\section{Introduction}

In the representation theory of finite-dimensional algebras, the global dimension, denoted as $\gldim$, is an important homological invariant that measures the homological complexity of the module category of an algebra, see for example \cite[etc]{CE1956,R1979}. Intuitively, it defines the supremum of the projective dimensions of all modules, thereby reflecting how “far” the algebra is from being semisimple (which corresponds to global dimension zero). Whether the global dimension of an algebra is finite or not directly governs the structure of its derived category and the complexity of computing its $\Ext$ and $\Tor$ functors. Therefore, characterizing the global dimension of various classes of algebras and studying their homological properties constitutes a fundamental pursuit in representation theory.

Gentle algebras, a special class of string algebras, have become a focal point of research due to their clear combinatorial definition and rich homological behavior. They were first introduced by Assem and Skowro\'{n}ski in the study of derived equivalences \cite{AS1987}, and owing to the foundational work of Wald, Waschb\"{u}sch, Butler and Ringel, their indecomposable modules can be completely classified as those constructed from strings and bands \cite{WW1985,BR1987}. More profoundly, Gei\ss~ and Reiten proved that all gentle algebras are Gorenstein algebras in \cite{GR2005}, meaning their self-injective dimension is also finite. This establishes a close connection between their finitistic dimension and the lengths of maximal forbidden paths.
Regarding the global dimension of gentle algebras, an elegant and profound characterization has been established: A gentle algebra has finite global dimension if and only if its bound quiver contains no ``forbidden cycles'' \cite[etc]{CPS2020, LGH2024, FOZ2024, Chang2025}, and the value of its global dimension is precisely equal to the supremum of the lengths of all forbidden paths in that quiver. This result tightly links the homological property of the algebra (global dimension) with its combinatorial structure (forbidden paths).

A fundamental objective in the representation theory of finite-dimensional algebras is the classification and understanding of algebras with well-behaved homological properties. Among these, quasi-tilted algebras, introduced by Happel, Reiten, and Smal{\o}, form a particularly important class. Here, an algebra $A$ is {\defines quasi-tilted} if its global dimension is at most 2 and, crucially, every indecomposable module $M$ satisfies either $\pdim M \=< 1$ or $\idim M \=< 1$ \cite{HRS1994}.
This homological condition implies a deep structural symmetry in its derived category, placing it between the classes of tilted and hereditary algebras. Therefore,
{\it finding effective criteria to determine whether a given algebra is quasi-tilted remains a significant problem.}
Huard and Liu first provide a sufficient and necessary condition for string algebra to be quasi-tilted by using a special quiver $\widetilde{\mathbb{A}}_{n,r,t}$, see \cite[Definition 3.5 and Theorem 3.6]{HL2000}. Later, Coelho and Tosar provide another description for gentle algebra to be quasi-tilted by using the cohomological widths of all indecomposable objects in the derived category, see \cite[Theorem 7]{CT2009}.
In this paper, we will consider the following question from {\it the perspective of quivers and combinations}.

\begin{question}
Let $A$ be a non-hereditary gentle algebra. For all indecomposable right $A$-modules $M$ with finite projective dimension and finite injective dimension, under what conditions does
\begin{align}\label{quest}
  \Hbound A := \sup_{ \bsm M\in\ind(\modcat A)  \\ \idim M + \pdim M < \infty \esm } (\pdim M+\idim M) \=< 2\cdot \gldim A -1
\end{align}
hold? {\rm(}Here, $\modcat A$ is the finitely generated right $A$-module category and $\ind(\modcat A)$ is the set of all isoclasses of indecomposable right $A$-module, and $\Hbound A$ is called the {\defines homological bound} of $A$.{\rm)}
\end{question}
\noindent Clearly, when $\gldim A=2$, (\ref{quest}) provides a characterization of the quasi-tilted property for non-hereditary gentle algebras since for each non-projective and non-injective indecomposable module $M$, $\pdim M + \idim M \=<3$ follows one of $\pdim M \=<1$ and $\idim M \=<1$ holds.

The main results of this paper are as follows:

\begin{enumerate}[label={\rm(\arabic*)}]
  \item We establish the following theorem.

\begin{theorem}
If the bound quiver $(\Q,\I)$ of a gentle algebra $A$ satisfies a certain local combinatorial condition —— namely, that the starting (or ending) vertices of all maximal forbidden paths of length $\>=2$ are strong sources (or strong sinks, see Section \ref{sect:pdim+idim}) —— then:
\begin{itemize}
  \item[\rm(a)]{\rm(Theorem \ref{thm:main 1})}
    the desired optimized bound (\ref{quest}), i.e.,
    \[ \Hbound A = \sup_{M\in\ind(\modcat A)} (\pdim M+\idim M) \=< 2\cdot \gldim A -1 \]
    holds in the case for $\gldim A < \infty$;
  \item[\rm(b)]{\rm(Theorem \ref{thm:main 2})}
    the desired optimized bound (\ref{quest}) in a finitistic dimensional version, i.e.,
    \[ \Hbound A = \sup_{\bsm M\in\ind(\modcat A)  \\ \idim M + \pdim M < \infty \esm} (\idim M + \pdim M) \=< 2\cdot \fdim A -1. \]
    holds for the case for $\gldim A = \infty$. {\rm(}Here, $\fdim A$ is the finitistic dimensional of $A$.{\rm)}
\end{itemize}
\end{theorem}
  \item For $\gldim A = 2$, the above results simplifies to $\pdim M + \idim M \=< 3$ for all indecomposable right $A$-module $M$, providing a sufficient condition for $A$ to be quasi-tilted (see Corollaries \ref{coro:quasitilted} and \ref{coro:quasitilted}).
  \item For any gentle algebra $A$ with $\gldim A = 2$, we establish a corollary that provide a sufficient and necessary condition for $A$ to be quasi-tilted by using the bound quiver of $A$:

\begin{corollary}[{\rm Corollary \ref{coro:main 3}}]
Let $A=\kk\Q/\I$ be a gentle algebra with $\gldim A = 2$. Then $A$ is quasi-tilted if and only if all strings in its bound quiver $(\Q,\I)$ satisfy one of the four conditions \ref{QT 1}, \ref{QT 2}, \ref{QT 3} and \ref{QT 4}.
\end{corollary}
Here, conditions \ref{QT 1}, \ref{QT 2}, \ref{QT 3} and \ref{QT 4} are written in the subsection \ref{subsect:suff and necess} of Section \ref{sect:QT}, they are the special cases of the pictures (1), (2), (3), and (4) shown in \Pic \ref{fig:QT str}, respectively.
\end{enumerate}

Finally, we also provided some important examples in this paper, see Section \ref{sect:examp}.

\section{Gentle algebras}

Let $\Q=(\Q_0,\Q_1,\s,\t)$ be a finite quiver, where $\Q_0$ and $\Q_1$ are vertex set and arrow set, respectively,
and $\s,\t:\Q_1\to \Q_0$ are functions sending each arrow $a\in\Q_1$ to its starting point and ending point, respectively.
In this paper, we put a composition of two arrows $a,b\in\Q_1$ is $ab$ if $\t(a)=\s(b)$, cf. \cite[Chap II]{ASS2006}.

\subsection{Gentle pairs and gentle algebras}

A bound quiver $(\Q, \I)$, i.e., a pair of quiver $\Q$ and an admissible ideal $\I$ of $\kk\Q$,
is called a {\defines gentle pair} if:
\begin{enumerate}
[label=(G\arabic*)]
  \item for each vertex in $\Q_0$, it is the source of at most two arrows and the target of at most two arrows; \label{G1}

  \item for each arrow $a\in\Q_1$, there is at most one arrow $b\in\Q_1$ such that $ab\notin\I$,
    and there is at most one arrow $c$ such that $ca\notin\I$; \label{G2}

  \item for each arrow $a\in\Q_1$, there is at most one arrow $b\in\Q_1$ such that $ab\in\I$,
    and there is at most one arrow $c$ such that $ca\in\I$; \label{G3}

  \item and $\I$ is generated by some paths of length two. \label{G4}
\end{enumerate}

\begin{definition} \rm
A {\defines gentle algebra} is a finite-dimensional algebra $A=\kk\Q/\I$ whose bound quiver $(\Q,\I)$ is a gentle pair.
\end{definition}

In \cite{AS1987}, Assem and Skowr\`{o}nski introduced gentle algebra, which is used to study the derived equivalence of hereditary algebras of Euclidean type $\w{\mathbb{A}}$.
All indecomposable modules over a gentle algebra have been described by Butler and Ringel
since all gentle algebras are string, see \cite[Section 3, page 161]{BR1987}.
We will recall the results in \cite{BR1987} in the next subsection.

\subsection{Finitely generated module categories for gentle algebras}

In order to facilitate the readers, we will review strings and bands here.
First, for each arrow $a$, we define it has a {\defines formal inverse} $a^{-1}$ and denote by $\Q_1^{-1}$ the set of all formal inverses.
Furthermore a path $a_1a_2\cdots a_n$ of length $n \>=1$ has a formal inverse $a_n^{-1}\cdots a_2^{-1}a_1^{-1}$,
and in particular, we define the formal inverse of each path $\e_v$ ($v\in\Q_0$) of length zero corresponding to $v$ is itself.
Naturally, we define $(a^{-1})^{-1}=a$ for any $a\in\Q_1\cap \Q_1^{-1}$.
Clearly, $\s$ and $\t$ have natural extension $\s, \t: \Q_1\cup \Q_1^{-1} \to \Q_0$ such that
\[ \s(a^{-}) = \t(a) \text{~and~} \t(a^-) = \s(a) \]
Strings and bands are sequences of arrows and formal inverses satisfying some conditions,
which were originally introduced by Wald and Waschb\"{u}sch in \cite{WW1985} and were called by V-sequences and primitive V-sequences.
Strings and bands can be used to describe indecomposable modules over any string and gentle algebra.
Next, we recall the definition of a string and band on a gentle pair.

\begin{definition}[{Strings and bands \cite{BR1987}}] \rm
Let $A=\kk\Q/\I$ be a gentle algebra.
\begin{itemize}
\item[(1)]
A {\defines string} in the gentle pair $(\Q,\I)$ of length $n$ is a sequence $\str = a_1\cdots a_n$ in $\Q_1\cup \Q_1^{-1}$ such that the following conditions hold.
\begin{enumerate}
[label=($\mathbf{S}$\arabic*)]
  \item The equation $\t(a_i) = \s(a_{i+1})$ holds for all $1 \=< i \=< n-1$.
  \item The sequence $\str$ is without any relation,
    i.e., if $a_i, a_{i+1}\in\Q_1$, then $a_ia_{i+1}\notin \I$;
    and if $a_i^{-1}, a_{i+1}^{-1}\in\Q_1$, then $a_{i+1}a_i\notin \I$.
  \item If $a_i\in\Q_1$ and $a_{i+1}\in\Q_1^{-1}$, then $a_i \ne a_{i+1}^{-1}$;
    and if $a_i\in\Q_1^{-1}$ and $a_{i+1}\in\Q_1$, then $a_{i+1} \ne a_i^{-1}$.
\end{enumerate}
In particular, we define $\varnothing$ as a {\defines trivial string}.
Two strings $\str$ and $\str'$ is called {\defines equivalent} if $\str = \str'$ or $\str^{-1}=\str'$.
The set of all equivalent classes of strings in $(\Q,\I)$ is denoted by $\Str(A)$.
\item[(2)]
A {\defines band} in gentle pair $(\Q,\I)$ of length $n$ is a string $\band = a_1\cdots a_n$ in $\Q_1\cup \Q_1^{-1}$ such that the following conditions hold.
\begin{enumerate}
[label=($\mathbf{B}$\arabic*)]
  \item $\t(a_n)=\s(a_1)$.
  \item $\band$ is not a power of any string.
  \item $\band^2$ is a string.
\end{enumerate}
Two band $\band$ and $\band'$ is called {\defines equivalent} if there is an integer $t$ with $0\=< t <n$ such that $\band[t] = \band'$ holds or $\band[t]^{-1}=\band'$ holds, where $\band[t] := b_{1+t}b_{2+t}\cdots b_nb_1\cdots b_{t}$.
The set of all equivalent classes of bands in $(\Q,\I)$ is denoted by $\Band(A)$.
\end{itemize}
\end{definition}

We use $\modcat A$ to represent the finitely generated right $A$-module category of $A$,
and use $\ind(\modcat A)$ to represent the set of all isoclasses of indecomposable modules.
Moreover, for simplicity, we do not differentiate between equivalent strings/bands,
and we call a string module is a {\defines direct string module} if the string $\str$ corresponded by it is a direct string,
where {\defines direct string} is either a trivial string, a string with length zero, or a string of the form
\begin{center}
  $a_1a_2\cdots a_n = \xymatrix{\bullet \ar[r]^{a_1} & \bullet \ar[r]^{a_2} & \cdots \ar[r]^{a_n} & \bullet}
      (=a_n^{-1}\cdots a_2^{-1}a_1^{-1})$.
\end{center}

\begin{theorem}[Wald--Waschb--Butler--Ringel \cite{WW1985, BR1987}] \label{thm:WWBR}
Let $A$ be a gentle algebra. Then there is a bijection
\[ \M : \Str(A) \cup (\Band(A)\times \mathscr{J}) \to \ind(\modcat A), \]
where $\mathscr{J}$ is the set of all Jordan block
\[ \pmb{J}_n(\lambda) :=
\left(\begin{matrix}
1 & \lambda & & \\
& 1 & ... & \\
& & ... & \lambda \\
& & & 1
\end{matrix}\right)_{n\times n}
\]
with non-zero eigenvalue $\lambda$.
The indecomposable module $\M(\str)$ corresponded by the string $\str \in \Str(A)$ is called a {\defines string module},
and the indecomposable module $B(n,\lambda) := \M(\band, \pmb{J}_n(\lambda))$ corresponded by the pair $(\band,\pmb{J}_n(\lambda)) \in \Band(A)\times \mathscr{J}$ is called a {\defines band module}.
\end{theorem}

\section{Proj.dim+inj.dim for indecomposable modules over gentle algebra} \label{sect:pdim+idim}

Let $A$ be a gentle algebra. We consider the upper bound of $\pdim M + \idim M$ in this section ($M\in \ind(\modcat A)$).
The following terminologies will be used.
\begin{itemize}
  \item A {\defines strong source} (resp., {\defines strong sink}) $v$ of a quiver $\Q$ is a source (resp., sink) of the quiver $\Q$ such that there at most one arrow $\alpha$ with $\s(\alpha)=v$ (resp., $\t(\alpha)=v$).

  \item A vertex $v$ on a string $\str$ (resp., band $\band$) is said to be a {\defines source} if there is no arrow on $\str$ (resp., $\band$) ending at $v$,
and dually a vertex $w$ on $\str$ (resp., band $\band$) is said to be a {\defines sink} if there is no arrow on $\str$ (resp., band $\band$) starting at $w$.

  \item A {\defines forbidden path} mentioned in the following statement is either a path $F=a_1\cdots a_n$ ($a_1, \ldots, a_n\in\Q_1$) with $a_ia_{i+1}\in\I$ ($1\=< i \=< n-1$), a path of length one, or a path $\e_v$ of length zero with $\sharp\s^{-1}(\e_v)=\sharp\{b\}=1=\sharp\t^{-1}(\e_v)=\sharp\{a\}$ and $ab\in\I$,
a {\defines left $($reps., right$)$ maximal forbidden path} $F=a_1\cdots a_n$ is a forbidden path such that there is no arrow $\alpha$ with $\t(\alpha)=\s(F)$ (resp., with $\t(F)=\s(\alpha)$) satisfying $\alpha a_1\in\I$ (resp., $a_n\alpha\in \I$).
The following two forbidden paths are called {\defines maximal forbidden path}:
\begin{itemize}
  \item a forbidden path that is both left maximal and right maximal.
  \item a forbidden path of length zero.
\end{itemize}
The name ``maximum forbidden path'' originates from \cite{OPS2018} and was originally referred to as the forbidden thread in \cite{AG2008}.
\end{itemize}

\subsection{Projective/injective dimensions of string modules} \label{subsect:str}

We use $\top(\str)$ to represent the set of all sources of $\str$ since the simple module $S(v)$ corresponded by source $v$ is a direct summand of the top $\top\M(\str)$ of $\M(\str)$, and furthermore, one can check
\[ \top\M(\str) \cong \bigoplus_{v\in \top(\str)} S(v). \]
Dually, we use $\soc(\str)$ to represent the set of all sinks of $\str$, and one can check
\[ \soc\M(\str) \cong \bigoplus_{v\in \soc(\str)} S(v). \]
The above two facts have been pointed out in \cite[Proposition 4.3]{LGH2024} and \cite[Lemma 3.4]{ZhangLiu2024} by using the geometric modules provided by Opper--Plamondon--Schroll \cite{OPS2018} and Baur--Coelho-Simo\~{e}s \cite{BCS2021}.
Moreover, in \cite[Theorem 2.8, or see Lemmas 2.4-2.7]{CPS2020} \c{C}anak\c{c}\i--Pauksztello--Schroll described the Homologies of all indecomposable objects in the derived category of gentle algebra by using a combined method.
Furthermore, we have the following important result.

\begin{lemma} \label{lemm:str}
For any string module $\M(\str)$,
\begin{enumerate}[label={\rm(\arabic*)}]
  \item its $1$-syzygy, write it as $\Omega_1(\M(\str))$, satisfies
\[ \Omega_1(\M(\str)) \cong L \oplus \bigg(\bigoplus_{v\in\soc(\str) \backslash \{\s(\str),\t(\str)\}} P(v) \bigg) \oplus R, \]
where $L$ and $R$ are direct string modules;
\label{lemm-str statem 1}
  \item and its $1$-cosyzygy, write it as $\mho_1(\M(\str))$, satisfies
\[ \mho_1(\M(\str)) \cong L' \oplus \bigg(\bigoplus_{v\in\top(\str) \backslash \{\s(\str),\t(\str)\}} E(v) \bigg) \oplus R', \]
where $L'$ and $R'$ are direct string modules.
\label{lemm-str statem 2}
\end{enumerate}
\end{lemma}

For the convenience of readers, we still provide a proof of Lemma \ref{lemm:str}.

\begin{proof}
Let $\str$ be a string of length $\>=1$ in a gentle pair $(\Q,\I)$. We ignore the case for length zero in this proof.
Each string can be written in one of the following forms:
\begin{align*}
\str_{\rmI}   = & \xymatrix{ v_1 \ar[r] & v_2 \ar@{~}[rr] && v_{n-1} \ar[r] & v_n } \\
\str_{\rmII}  = & \xymatrix{ v_1 \ar@{<-}[r] & v_2 \ar@{~}[rr] && v_{n-1} \ar@{<-}[r] & v_n }  \\
\str_{\rmIII} = & \xymatrix{ v_1 \ar[r] & v_2 \ar@{~}[rr] && v_{n-1} \ar@{<-}[r] & v_n }  \\
\str_{\rmIV}  = & \xymatrix{ v_1 \ar@{<-}[r] & v_2 \ar@{~}[rr] && v_{n-1} \ar[r] & v_n }
\end{align*}
We consider the case of $\str=\str_{\rmI}$, the other cases are similar.
First of all, we rewrite $\str_{\rmI}$ in the following form
($u\>=1$, and each arrow ``$\xymatrix{\ar@{~>}[r] &}$'' is a path of length $\>= 1$)
\[
\xymatrix@C=0.8cm@R=0.6cm{
t_1 \ar[rd]^{a} && && && && && \\
& \bullet \ar@{~>}[rd]
& & t_2 \ar@{~>}[ld] \ar@{~>}[rd]
& & \cdots \ar@{~>}[ld] \ar@{~>}[rd]
& & t_u \ar@{~>}[ld] \ar@{~>}[rd]^{\wp}
& \\
& & s_1
& & s_2
& & \cdots
& & s_u,
}
 \]
then $\top(\str_{\rmI}) = \{t_1, t_2, \ldots, t_u\}$ and $\soc(\str_{\rmI}) = \{s_1, s_2,\ldots, s_u\}$.
Thus, $\displaystyle \top \M(\str_{\rmI}) \cong \bigoplus_{i=1}^u S(t_i)$.
It follows that the projective cover of $\M(\str_{\rmI})$ is
\[ p: \bigoplus_{i=1}^u S(t_i) \to \M(\str_{\rmI}), \]
and we have $P(s_1)$, $\ldots$, $P(s_{u-1})$ are direct summand of $\Omega_1(\M(\str_{\rmI}))$,
cf. \Pic \ref{fig:proj cover of M(s)}.
\begin{figure}[htbp]
\centering
\includegraphics[width=17.5cm]{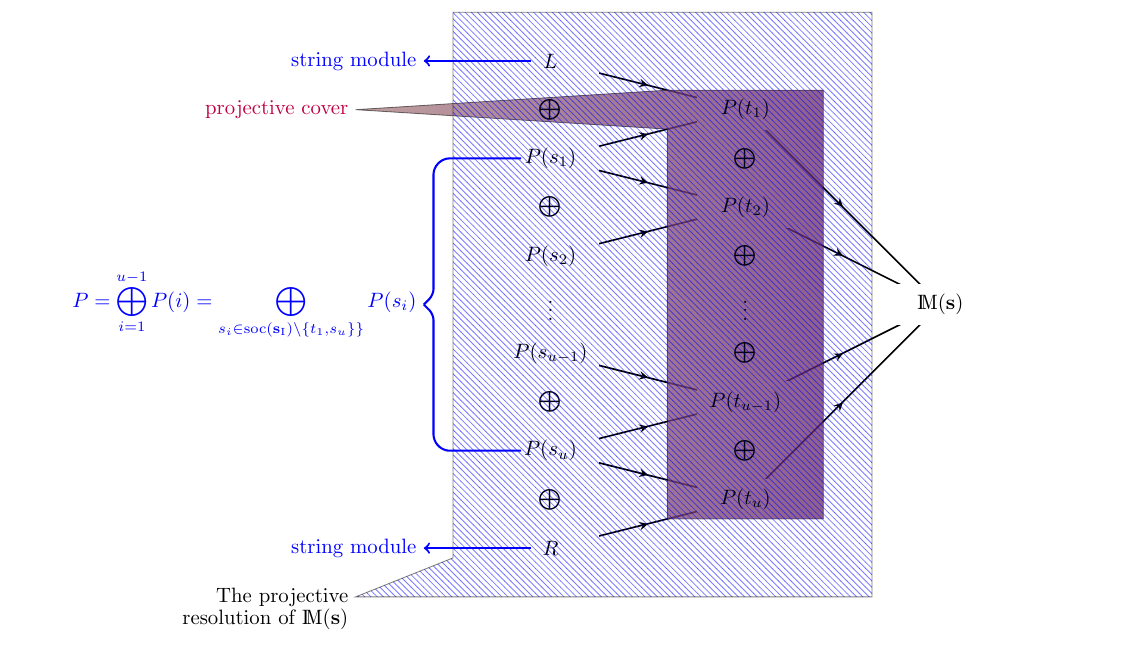}
\caption{\textsf{The projective cover of $\M(\str)$ in the case for $\str=\str_{\rmI}$ and its kernel}}
\label{fig:proj cover of M(s)}
\end{figure}
Notice that the 1-syzygy $\Omega_1(\M(\str_{\rmI}))$ has another direct summands as follows:
\begin{enumerate}
[label=Case \arabic*]
  \item if there exists an arrow $\alpha \ne a$ with $\s(\alpha)=t_1$,
    then $\Omega_1(\M(\str_{\rmI}))$ has a direct summand $L$ which is a direct string module with $\top L = S(\t(\alpha))$; \label{lemm-str:(1)}
  \item if there exists an arrow $\beta$ with $\s(\beta)=s_u$,
    then $\Omega_1(\M(\str_{\rmI}))$ has a direct summand $R$ which is a direct string module with $\top R = S(\t(\beta))$.
    \label{lemm-str:(2)}
\end{enumerate}
Thus, we have
\[\Omega_1(\M(\str_{\rmI})) = L \oplus \bigg( \bigoplus_{i=1}^{u-1} P(i) \bigg) \oplus R \]
as required, where $L$ and $R$ may be zero, and if $u=1$ then $P$ is zero.
\end{proof}

A vertex $v$ in gentle pair $(\Q,\I)$ is called a {\defines vertex with relation(s)} if there are two arrows $a_1$ and $a_2$ with $\t(a_1)=v=\s(a_2)$ such that $a_1a_2\in \I$ holds.
Conversely, we call $v$ is a {\defines vertex without relation on a sequence $\{a_u\}_{1\=< u\=< m} := a_1\cdots a_m$} (here, $a_1,\ldots, a_m\in \Q_1\cup\Q_1^{-1}$, and $\t(a_i)=\s(a_{i+1})$ for all $1\=< i < m$)
if $v \in \{\s(a_1),\ldots, \s(a_m), \t(a_m)\}$ is not a vertex with a relation on $\{a_u\}_{1\=< u\=< m}$.
To be precise, $v$ satisfies one of the following conditions.
\begin{itemize}
  \item $\t(a_i)=v=\s(a_{i+1})$ for some $1\=<i<m$, $a_i,a_{i+1}\in \Q_1$ (resp., $\in\Q_1^{-1}$), and $a_ia_{i+1}$ (resp., $a_{i+1}^{-1}a_i^{-1}$) $\not\in \I$;
  \item $\t(a_i)=v=\t(a_{i+1})$ for some $1\=<i<m$, $a_i\in \Q_1$ (resp., $\in\Q_1^{-1}$), and $a_{i+1}\in \Q_1^{-1}$ (resp., $\in\Q_1$);
  \item $v \in \{\s(a_1),\t(a_m)\}$.
\end{itemize}
The terminology ``vertex without relation on $\{a_u\}$'' is used in the discourse of Proposition \ref{prop:idim/pdim} and the proof of Lemma \ref{lemm:QT}.
Notice that we have the following three facts:
\begin{itemize}
  \item $L$ is not projective if and only if the vertex $\t(\alpha)$ is a vertex with a relation,
where $\alpha$ is the arrow given in \ref{lemm-str:(1)};
  \item $R$ is not projective if and only if the vertex $\t(\beta)$ is a vertex with a relation,
where $\beta$ is the arrow given in \ref{lemm-str:(2)};
  \item $L$ and $R$ are direct string modules such that the strings $\M^{-1}(L)=\mathbf{l}$ and $\M^{-1}(R)=\mathbf{r}$ both are direct strings, i.e., $\mathbf{l}$ and $\mathbf{r}$ is of the form $\str_{\rmI}$.
\end{itemize}
Thus, we can consider the 1-syzygies of $L$ and $R$ by using Lemma \ref{lemm:str},
and obtain the 2-syzygy of $\M(\str)$. By repeatedly using Lemma \ref{lemm:str}, we have the following result.

\begin{proposition} \label{prop:idim/pdim}
Let $\str = a_1\cdots a_m$ $(a_1,\ldots, a_m \in \Q_1\cup\Q_1^{-1})$ be a string.
\begin{itemize}
  \item[\rm(1)]
    \begin{enumerate}[label={\rm(1.\arabic*)}]
      \item If there exists a left maximal forbidden path $F_1=f_1\cdots f_{u_{\Left}}$
        $(f_1, \ldots, f_{u_{\Left}}$ $\in \Q_1)$ such that $\t(F_1)=\s(\str)$ holds
        and $\t(f_{u_{\Left}}) = \s(a_1)$ is a vertex without relation on the sequence
        $\{f_1,\ldots, f_{u_{\Left}},a_1,\cdots,a_m\}$,
        then $\idim \M(\str)\>= u_{\Left}$. \label{prop:string (1) 1}

      \item If there exists a left maximal forbidden path $F_2=g_1\cdots g_{u_{\Right}}$
       $(g_1,\ldots, g_{u_{\Right}})$ such that $\t(F_2)=\t(\str)$ holds
       and $\t(a_m)=\t(g_{u_{\Right}})$  is a vertex without relation on the sequence
       $\{a_1,\ldots, a_m, g_{u_{\Right}}^{-1}, \ldots, g_1^{-1}\}$,
       then $\idim \M(\str)\>= u_{\Right}$. \label{prop:string (1) 2}
      \item If \ref{prop:string (1) 1} and \ref{prop:string (1) 2} hold, then $\idim \M(\str) = \max\{u_{\Left}, u_{\Right}\}$.
    \end{enumerate}
  \item[\rm(2)]
    \begin{enumerate}[label={\rm(2.\arabic*)}]
      \item If there exists a right maximal forbidden path $F_1=f_1\cdots f_{d_{\Left}}$
        $(f_1, \ldots, f_{d_{\Right}}$ $\in \Q_1)$ such that $\s(F_1)=\s(\str)$ holds
        and $\s(f_1)=\s(a_1)$ is a vertex without relation on the sequence
        $\{ f_{d_{\Left}}^{-1}, \ldots, f_1^{-1}, a_1,\ldots, a_m\}$,
        then $\pdim \M(\str)\>= d_{\Left}$. \label{prop:string (2) 1}
      \item If there exists a left maximal forbidden path $F_2=g_1\cdots g_{d_{\Right}}$
        $(g_1, \ldots, g_{d_{\Right}}$ $\in \Q_1)$ such that $\s(F_1)=\t(\str)$ holds
        and $\s(g_1)=\t(a_m)$ is a vertex without relation on the sequence
        $\{ a_1,\ldots, a_m, g_1,\ldots, g_m\}$,
        then $\pdim \M(\str)\>= d_{\Right}$. \label{prop:string (2) 2}
      \item If \ref{prop:string (2) 1} and \ref{prop:string (2) 2} hold, then $\pdim \M(\str) = \max\{d_{\Left}, d_{\Right}\}$.
    \end{enumerate}
\end{itemize}
\end{proposition}

\subsection{Projective/injective dimensions of band modules} \label{subsect:band}

For each module $B(n,\lambda)$ ($n\in \NN$, and if $n=0$, then $B(0,\lambda)$ is not a band and we call it a trivial case; and if $n>0$, then $B(n,\lambda)$ is a band module), we have an Auslander--Reiten sequence of the form
\[ 0 \To{} B(n,\lambda) \To{} B(n+1,\lambda) \oplus B(n-1,\lambda) \To{} B(n,\lambda) \To{} 0. \]
Thus, we have a short exact sequence
\[ 0 \To{} B(1,\lambda) \To{} B(n, \lambda) \To{} B(n-1,\lambda) \To{} 0 \]
for any $n\>= 2$. Obviously, $\pdim B(1,\lambda)=1$ and $\pdim B(n-1,\lambda) = 1$ admit $\pdim B(n, \lambda)$ $=1$ by using Horseshoe Lemma and the fact that band modules are not projective.
Furthermore, if $\pdim B(1,\lambda)=1$, then we can prove the following proposition by using induction.

\begin{proposition} \label{prop:band}
The projective and injective dimension of any band module over a gentle algebra is $1$.


\end{proposition}

In \cite[Corollary 2.13]{CPS2020}, the authors have shown that Proposition \ref{prop:band} holds in the case of band modules to be quasi-simple. Here, {\defines quasi-simple band module} is a band module of the form $B(1,\lambda)$. For the convenience of readers, we still provide a proof of Proposition \ref{prop:band}.

\begin{proof}
Let $B(1,\lambda)$ be a band module and $\band$ be the band corresponding to it.
Next, we show $\pdim B(1,\lambda) =1$, the proof of $\pdim B(n,\lambda) =1$ can be proved by induction
and the extension $0 \To{} B(1,\lambda) \To{} B(n, \lambda) \To{} B(n-1,\lambda) \To{} 0$.
Since the band $\band$ can be written as
\[
\xymatrix@C=0.8cm@R=0.6cm{
 & & & & t_1 \ar@{~>}@/^2pc/[rrrdd] \ar@{~>}@/_2pc/[llldd]
 & & & \\
 & & t_2 \ar@{~>}[ld] \ar@{~>}[rd]
& & \cdots \ar@{~>}[ld] \ar@{~>}[rd]
& & t_u \ar@{~>}[ld] \ar@{~>}[rd]_{\wp}
& \\
  & s_1
& & s_2
& & \cdots
& & s_u,
}
 \]
we obtain that the projective cover of $B(1,\lambda) \cong \M(\band, \pmb{J}_1(\lambda))$ is
\[ p: \bigoplus_{i=1}^u P(t_i) \to B(1,\lambda). \]
It follows that the $1$-syzygy is
\[ \Omega_1(B(1,\lambda)) = \bigoplus_{i=1}^u P(s_i), \]
cf. \Pic \ref{fig:proj cover of B}. So, $\pdim B(1,\lambda) =1$.

We can show that the injective dimension of any band module is one by a dual way.
\begin{figure}[htbp]
\begin{center}
\includegraphics[width=10cm]{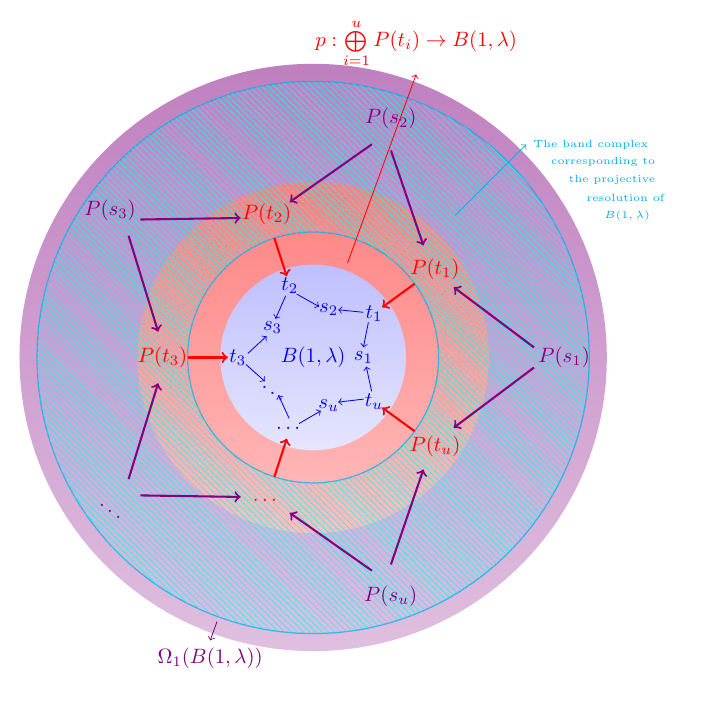}
\caption{\textsf{The projective resolution of a band module $B(1, \lambda)$}}
\label{fig:proj cover of B}
\end{center}
\end{figure}
\end{proof}

\subsection{The case for gentle algebra with finite global dimension}

Recall that a {\defines forbidden cycle} is an oriented cycle $\C = c_1\cdots c_n$ ($\t(c_n)=\s(c_1)$) such that $c_1c_2, \ldots$, $c_{n-1}c_n$, $c_nc_1$ $\in \I$. If a gentle pair $(\Q,\I)$ has a forbidden cycle, then $A=\kk\Q/\I$ has an infinite global dimension since there exists an arbitrarily length forbidden path on $\C$.
The following result describes the global dimension of gentle algebra.

\begin{theorem}[\!\!{\cite[Theorems 5.10 and 6.9]{LGH2024}}] \label{thm:LGH2024}
A gentle algebra has a finite global dimension if and only if its bound quiver does not have any forbidden cycle.
A gentle algebra has a global dimension $n$ if and only if the supremum of the lengths of all forbidden paths in its bound quiver is $n$.
\end{theorem}

Next, we show the following key lemma by using Theorem \ref{thm:LGH2024}.

\begin{lemma} \label{lemm:Forbid}
Let $M \in \Ima(\M|_{\Str(A)})$ be a string module over gentle algebra $A$,
and $\str_M=\M^{-1}(M) = s_1\cdots s_t$ $(s_1, \ldots, s_t \in\Q_1\cup\Q_1^{-1})$ be the string corresponded by it.
If $\gldim A < \infty$, then the following statements hold.
\begin{itemize}
  \item[\rm(1)] There is at most one left maximal forbidden path $F=f_1\cdots f_{u_{\Left}}$ $($resp., $F'=f_1'\cdots f_{u_{\Right}}'$$)$ $(f_1,\ldots, f_{u_{\Left}}, f_1',\ldots, f_{u_{\Right}}'\in\Q_1)$ such that $\t(F) = \s(\str_M)$ $($resp., $\t(F)=\t(\str_M)$$)$ and $f_1s_1$ $($resp., $f_{u_{\Right}}s_t^{-1}$$)$ is a string of length two;
  \item[\rm(2)] There is at most one right maximal forbidden path $\w{F}=\w{f}_1\cdots \w{f}_{d_{\Left}}$ $($resp., $\w{F}=\w{f}_1'\cdots \w{f}_{d_{\Right}}'$$)$ $(\w{f}_1,\ldots, \w{f}_{d_{\Left}}, \w{f}_1',\ldots, \w{f}_{d_{\Right}}\in\Q_1)$ such that $\s(\w{F}) = \s(\str_M)$ $($resp., $\s(\w{F})=\t(\str_M)$$)$ and $\w{f}_1^{-1}s_1$ $($resp., $s_tf_1$$)$ is a string of length two.
\end{itemize}
\end{lemma}

\begin{proof}
By the definition of gentle algebra, if $s_1\in\Q_1$, then we can find at most two arrows $\alpha$ and $\beta$ with $\t(\alpha) = \s(s_1)$ and $\t(\beta)=\s(s_1)$. It follows that if $\alpha s_1 \notin \I$ (resp., $\alpha s_1 \in \I$) then $\beta s_1 \in \I$ (resp., $\beta s_1 \notin \I$). Then we obtain a unique forbidden path $f_{u_{\Left}}=\beta$ (resp., $\alpha$) of length one such that $\t(f_{u_{\Left}}) = \s(s_1) = \s(\str_M)$.
By using the definition of gentle algebra again, the exists at most one arrow $f_{u_{\Left}-1}$ with $\t(f_{u_{\Left}-1}) = \s(f_{u_{\Left}})$ and $f_{u_{\Left}-1}f_{u_{\Left}}\in\I$. Thus, by using $\gldim A<\infty$ and Theorem \ref{thm:LGH2024},
we can find a unique left maximal forbidden path $F = f_1\cdots f_{u_{\Left}}$ such that $\t(F)=\s(s_1)$ and $f_{u_{\Left}}s_1\notin\I$ holds.
We can consider the case for $s_1\in\Q_1^{-1}$ in a similar way, and so the statement (1) holds.
One can check that (2) holds dually.
\end{proof}

\begin{corollary} \label{coro:pdim+idim}
Keep the notations from Lemma \ref{lemm:Forbid}, we have
\[ \pdim M + \idim M \mathop{=}\limits^{\spadesuit}
     \max\{u_X + d_X \mid X \in \{\Left,\Right\}\}
   ~(\mathop{\leqslant}\limits^{\heart} 2\cdot \gldim A), \]
where $\gldim A$ is the global dimension of gentle algebra $A$.

$($Note that this corollary holds for $\gldim A=\infty$ if the string $\str_M$ corresponded by $M$ satisfying the statements (1) and (2) given in Lemma \ref{lemm:Forbid}.$)$
\end{corollary}

\begin{proof}
The equality $\spadesuit$ is a direct corollary of Lemma \ref{lemm:Forbid} and Proposition \ref{prop:idim/pdim},
and the inequality $\heart$ is trivial since $\pdim M$ and $\idim M$ are less than or equal to $\gldim A$.
\end{proof}

Now we provide the first main result of this paper.

\begin{theorem} \label{thm:main 1}
Let $A=\kk\Q/\I$ be a non-hereditary gentle algebra with a finite global dimension. If one of the following conditions holds:
\begin{enumerate}[label={\rm(\arabic*)}]
  \item the starting vertices of all maximal forbidden paths of length $\>= 2$ in $(\Q,\I)$ are strong sources of $\Q$;
    \label{thm-main:(1)}
  \item the ending vertices of all maximal forbidden paths of length $\>= 2$ in $(\Q,\I)$ are strong sinks of $\Q$,
    \label{thm-main:(2)}
\end{enumerate}
then
\begin{align}\label{formula:thm-main 1}
  \Hbound A \=< 2\cdot \gldim A -1.
\end{align}
\end{theorem}

\begin{proof}
Assume that $(\Q,\I)$ satisfies the condition \ref{thm-main:(1)}, then all strings can be divided to four cases as follows:
\begin{enumerate}[label={\rm(\alph*)}]
  \item a string $\str = \alpha\in\Q_1$ of length one which is the {\defines inner} of some maximal forbidden path $F=f_1\cdots f_{\ell}$ ($f_1,\ldots, f_{\ell} \in \Q_1$), i.e., $\ell \>=3$ and $\alpha \in \{f_2, \ldots, f_{\ell-1}\}$, if $\gldim A \>= 3$; \label{thm-main:pf case 1}
  \item a string $\str$ of length $\>= 1$ such that there is a maximal forbidden path $F=f_1\cdots f_{\ell}$ of length $\ell \>= 1$ such that $\s(\str)=\t(f_i)$ (resp., $\t(\str)=\t(f_j)$) holds for some $1\=< i \=< \ell$ (resp., $1\=< j \=< \ell$); \label{thm-main:pf case 2}
  \item a string $\str$ of length $\>= 1$ such that $\s(\str)$ and $\t(\str)$ are not on any forbidden path of length $\>= 2$; \label{thm-main:pf case 3}
  \item a string of length zero. \label{thm-main:pf case 4}
\end{enumerate}

In the case \ref{thm-main:pf case 1}, we assume $\str=\alpha$ is of the following form
\[ \xymatrix{
\bullet \ar@{~>}[rrd]^{f_1\cdots f_{i-1}}  &&&&& &\\
&& \bullet \ar[r]^{\alpha=f_i}
& \bullet \ar@{~>}[rrd]^{f_{i+1}\cdots f_{\ell}} && &= ~ F \text{~is forbidden.} \\
&&&&& \bullet & }  \]
In this case, $F' = f_1\cdots f_{i-1}$ is a unique left maximal forbidden path ending with $\s(\alpha)$,
and $F''=f_{i+1}\cdots f_{\ell}$ is a unique right maximal forbidden path starting with $\t(\alpha)$.
By using the notations given in Lemma \ref{lemm:Forbid}, we have $u_{\Left}=0$, $d_{\Left}\=< 1$, $u_{\Right}=0$, and $d_{\Right}\=< 1$, it follows that
\[ \pdim \M(\alpha) + \idim \M(\alpha) \=< 1 \=< 2\cdot \gldim A -1  \]
as required by Corollary \ref{coro:pdim+idim} and $\gldim A \>= 3$.

In the case \ref{thm-main:pf case 2}, we assume that $F=f_1\cdots f_{\ell}$ and $\str=s_1\cdots s_t$ ($t\>=1$) are of the following form
\[ \xymatrix{
     v \ar@{~>}[rr]^{f_1\cdots f_{i}}
  && \bullet \ar@{~>}[dd]^{\str=s_1\cdots s_t} \ar@{~>}[rr]^{f_{i+1}\cdots f_{\ell}}
  && \bullet, \\
  &&  && \\ && \bullet &&} \]
where $v$ is a strong source of $\Q$, then there may be a maximal forbidden path $F'=f_1'\cdots f_l'$ of length $1\=< l \=< \gldim A$ such that $\t(\str) \in \{\t(f_1'), \ldots, \t(f_l')\}$. Suppose $\t(\str) = \t(f_j')$.
Then by Lemma \ref{lemm:Forbid}, we have $u_{\Left} = i (\>=1)$, $d_{\Left}=\ell-i$, $u_{\Right} = j$, and $d_{\Right} = l-j$.
In this case, Corollary \ref{coro:pdim+idim} admits
\[ \pdim \M(\str) + \idim \M(\str) \=< \max\{i+\ell -i=\ell, i+l-j, j+\ell-i, j+(l-j)=l\}.  \]
Notice that $\ell, l \=< \gldim A$, we obtain
\begin{align}
  & i+\ell -i=\ell \=< \gldim A < 2\cdot \gldim A -1 \label{formula in thm-main 1} \\
  & i+l-j \=< l+l-1=2l-1 \=<2\cdot \gldim A-1; \label{formula in thm-main 2} \\
  & j+\ell-i\=< l+\ell-1 \=< 2\cdot \gldim A-1; \label{formula in thm-main 3} \\
  & j+(l-j)=l \=<\gldim A < 2\cdot \gldim A-1. \label{formula in thm-main 4}
\end{align}
Equations (\ref{formula in thm-main 1})--(\ref{formula in thm-main 4}) admit
$\pdim \M(\str) + \idim \M(\str) \=< 2\cdot \gldim A-1$ as required.

In the case \ref{thm-main:pf case 3}, we have $\pdim \M(\str) \=< 1$ and $\idim \M(\str) \=< 1$ by Proposition \ref{prop:idim/pdim}, immediately.

In the case \ref{thm-main:pf case 4}, we have $\str = \e_v$ for some $v\in \Q_0$ and $\M(\e_v) = S(v)$ is the simple module correspond by $v\in\Q_0$. If $v$ is not a vertex on any maximal forbidden path, then $\pdim S(v) \=< 1$ and $\idim S(v) \=< 1$, it follows that $\pdim S(v) + \idim S(v) \=< 2 \=< 2\cdot \gldim A -1$ in the case for $\gldim A \>=2$.
If $v$ is a vertex on some maximal forbidden path $F=f_1\cdots f_{\ell}$, since $\s(F)$ is a source, we have $v \in \{\t(f_1),\ldots, \t(F_{\ell})\}$, and so, for the case $v = \t(f_i)$ ($1\=< i\=< \ell$),
we obtain $\pdim S(v)=\ell-i$ and $\idim S(v) = i$. Then $\pdim S(v)+\idim S(v) = \ell \=< \gldim A\=< 2\cdot \gldim A-1$.

Moreover, for any band module $B(n,\lambda)$, Lemma \ref{lemm:Forbid} follows $\pdim B(n,\lambda)+\idim B(n,$ $\lambda)=2 < 2\cdot \gldim A -1$. Therefore, \ref{thm-main:(1)} holds for all indecomposable modules.

We can prove \ref{thm-main:(2)} by using a dual way.
\end{proof}

Notice that if a gentle pair $(\Q,\I)$ satisfies neither the condition \ref{thm-main:(1)} nor the condition \ref{thm-main:(2)}, then $\idim M + \pdim M = 2\cdot \gldim A$ may hold, see Example \ref{examp:1} in Section \ref{sect:examp}.
A string $\str=a_1\cdots a_n$ is called a {\defines maximal string} if the following conditions hold.
\begin{itemize}
  \item $a_1\in\Q_1$ (resp., $a_1\in\Q_1^{-1}$) admits that any arrow $\alpha$ with $\t(\alpha)=\s(a_1)$ (resp., $\s(a_1)=\t(a_1^{-1})$) satisfies $\alpha a_1\in \I$ (resp., $a_1^{-1}\alpha \in \I$).
  \item $a_n\in\Q_1$ (resp., $a_n\in\Q_1^{-1}$) admits that any arrow $\beta$ with $\t(a_n)=\s(\beta)$ (resp., $\t(\beta)=\s(a_n^{-1})$) satisfies $a_n\beta \in \I$ (resp., $\beta a_n^{-1} \in \I$).
\end{itemize}
The following result provides another case such that (\ref{formula:thm-main 1}) holds.

\subsection{The case for gentle algebra with infinite global dimension}

In \cite{GR2005}, Gei\ss~ and Reiten proved that any gentle algebra are Gorenstein, i.e., gentle algebras have finite self-injective dimensions. It follows that self-injective dimension and finitistic dimension of gentle algebra coincide.
The following result provide a description of finitistic dimension of gentle algebra.

\begin{theorem}[Gei\ss--Reiten{\cite{GR2005}}] \label{thm:GR2005}
The finitistic dimension $\fdim A$ of a gentle algebra $A$ equals the supremum of the lengths of all maximal forbidden paths in $(\Q, \I)$.
\end{theorem}

Now, we show the second key lemma of our paper.

\begin{lemma}
Let $A=\kk\Q/\I$ be a gentle algebra with $\gldim A =\infty$.
Then for any string $\str = s_1\cdots s_n$ in $(\Q,\I)$ $(s_1,\ldots, s_n \in \Q_1)$ whose starting point $\s(\str)$ and ending point $\t(\str)$ do not on any forbidden cycle of $(\Q,\I)$, we have
\[ \pdim \M(\str) + \idim \M(\str) < \infty.\]
Furthermore, keep the notations from Lemma \ref{lemm:Forbid}, the above inequality can be written as
\[ \pdim \M(\str) + \idim \M(\str) \=<
  \max\{u_X + d_X \mid X \in \{\Left,\Right\}\}
  \=< 2\cdot \fdim A<\infty. \]
\end{lemma}

\begin{proof}
If there exists a forbidden path $F$ of length $\ell > \fdim A$ such that $\t(F) = \s(\str)$,
then $F$ is not left maximal by using the fact admitted by Theorem \ref{thm:GR2005} that
$\fdim A$ is greater than or equal to the supremum of the length of all left maximal forbidden paths.
Then $F$ must be a forbidden path on some forbidden cycle, say $\C$, of $(\Q,\I)$.
It follows that $\t(F) =\s(\str)$ is a vertex on $\C$, a contradiction.
Therefore, the length of any left maximal forbidden path $F$ with $\t(F)=\s(\str)$ must be $\=< \fdim A$.
One can check that the length of any right maximal forbidden path $F'$ with $\s(F)=\s(\str)$ must be $\=< \fdim A$,
and we can obtain the dual result for $\t(\str)$ by using a dual method.

Now, keep the notation from Lemma \ref{lemm:Forbid} for the string $\str_M$ of an indecomposable module $M$ such that
$\s(\str_M)$ and $\t(\str_M)$ do not on any forbidden cycle of $(\Q,\I)$,
then we have $u_{\Left}$, $d_{\Left}$, $u_{\Right}$, and $d_{\Right}$ are finite,
and they are less than or equal to $\fdim A$ by Theorem \ref{thm:GR2005}. Then we have
\[ \pdim M + \idim M \=< \max \{u_X+d_X \mid X\in\{\Left, \Right\}\} \=< 2\cdot \fdim A < \infty \]
whose proof is similar to that of Corollary \ref{coro:pdim+idim}.
\end{proof}

Now we give the second main result of this paper.

\begin{theorem} \label{thm:main 2}
Let $A=\kk\Q/\I$ be a gentle algebra with $\gldim A =\infty$ and $\fdim A\>=2$
such that one of the following statements holds.
\begin{enumerate}[label={\rm(\arabic*)}]
  \item the starting vertices of all maximal forbidden paths of length $\>= 2$ in $(\Q,\I)$ are strong sources of $\Q$;
    \label{thm-main 2:(1)}
  \item the ending vertices of all maximal forbidden paths of length $\>= 2$ in $(\Q,\I)$ are strong sinks of $\Q$,
    \label{thm-main 2:(2)}
\end{enumerate}
then
\[ 
\Hbound A \=< 2\cdot \fdim A -1. \]
\end{theorem}

\begin{proof}
Assume that $(\Q,\I)$ satisfies the condition \ref{thm-main:(1)}, then all strings can be divided to four cases as follows:
\begin{enumerate}[label={\rm(\alph*)}]
  \item a string $\str = \alpha\in\Q_1$ of length one which is the inner of some maximal forbidden path $F=f_1\cdots f_{\ell}$ ($f_1,\ldots, f_{\ell} \in \Q_1$); \label{thm-main 2:pf case 1}
  \item a string $\str$ of length $\>= 1$ such that there is a forbidden path $F=f_1\cdots f_{\ell}$ of length $\ell \>= 1$ such that $\s(\str)=\t(f_i)$ (resp.,$\t(\str)=\t(f_j)$) holds for some $1\=< i \=< \ell$ (resp., $1\=< j \=< \ell$); \label{thm-main 2:pf case 2}
  \item a string $\str$ of length $\>= 1$ such that $\s(\str)$ and $\t(\str)$ are not on any forbidden path of length $\>= 2$; \label{thm-main 2:pf case 3}
  \item a string of length zero (we consider only the case that the string is a vertex on some maximal forbidden path, since the projective dimension and injective dimension of the simple module corresponding to a vertex on some forbidden cycle are infinite). \label{thm-main 2:pf case 4}
\end{enumerate}

In the case \ref{thm-main 2:pf case 2}, we assume that $F=f_1\cdots f_{\ell}$ and $\str=s_1\cdots s_t$ ($t\>=1$) are of the following form
\[ \xymatrix{
    v_1 \ar@{~>}[rr]^{f_1\cdots f_{i}}
  && \bullet \ar@{~>}[dd]^{\str=s_1\cdots s_t} \ar@{~>}[rr]^{f_{i+1}\cdots f_{\ell}}
  && \bullet, \\
  &&  && \\ && \bullet &&} \]
where $v_1$ is a strong source of $\Q$, then there may be a forbidden path $F'=f_1'\cdots f_l'$ of length $1\=< l \=< \gldim A$ such that $\t(\str) = \t(f_1')$ holds. If $F'$ is a forbidden path on some forbidden cycle $\C$, then $\pdim \M(\str) = \infty$ and $\idim \M(\str) = \infty$. This is not the scenario to be considered in this proof.
Now, we assume that $F'$ is a maximal forbidden path of length $\>=2$, then $v_2=\s(F')$ is a source of $\Q$, and we have $\t(\str) \in \{\t(f_1'),\ldots, \t(f_l')\}$.
Suppose $\t(\str) = \t(f_j')$, then by Lemma \ref{lemm:Forbid}, we have $u_{\Left} = i (\>=1)$, $d_{\Left}=\ell-i$, $u_{\Right} = j$, and $d_{\Right} = l-j$.
In this case, Corollary \ref{coro:pdim+idim} admits
\[ \pdim \M(\str) + \idim \M(\str) \=< \max\{i+\ell -i=\ell, i+l-j, j+\ell-i, j+(l-j)=l\}.  \]
Notice that $F$ and $F'$ are maximal forbidden paths, then $\ell, l \=< \fdim A$. Thus, we obtain
\begin{align}
  & i+\ell -i=\ell \=< \fdim A < 2\cdot \fdim A -1 \label{formula in thm-main 2, 1} \\
  & i+l-j \=< l+l-1=2l-1 \=<2\cdot \fdim A-1; \label{formula in thm-main 2, 2} \\
  & j+\ell-i\=< l+\ell-1 \=< 2\cdot \fdim A-1; \label{formula in thm-main 2, 3} \\
  & j+(l-j)=l \=<\fdim A < 2\cdot \fdim A-1. \label{formula in thm-main 2, 4}
\end{align}
(\ref{formula in thm-main 2, 1})--(\ref{formula in thm-main 2, 4}) admits $\idim \M(\str) + \pdim \M(\str) \=< 2\cdot \fdim A -1$ as required.

The proofs of cases \ref{thm-main 2:pf case 1}, \ref{thm-main 2:pf case 3}, and \ref{thm-main 2:pf case 4} are similar to that of Theorem \ref{thm:main 1}'s cases \ref{thm-main:pf case 1}, \ref{thm-main:pf case 3}, and \ref{thm-main:pf case 4}.
Therefore, we obtain the statement \ref{thm-main 2:(1)}.

One can check that \ref{thm-main 2:(2)} holds by using a dual way.
\end{proof}


\section{Quasi-tilted gentle algebras} \label{sect:QT}

In this section, we provide two methods to judge gentle algebras to be quasi-tilted.

\subsection{A sufficient conditions}

Quasi-tilted algebra, a generalized tilted algebra, is introduced by Happel, Reiten, and Smal{\o} in \cite{HRS1994}. Its definition is written as follows:

\begin{definition} \rm
A finite-dimensional algebra $\alg$ is said to be {\defines quasi-tilted} if $\gldim\alg \=< 2$ and for any indecomposable module $M$, one of $\pdim M \=< 1$ and $\idim M \=< 1$ holds.
\end{definition}

Equivalently, for a finite-dimensional algebra $\alg$ with $\gldim\alg \=< 2$, $\alg$ is quasi-tilted if and only if $\pdim M + \idim M \=< 3$ holds for all non-projective indecomposable modules and non-injective indecomposable modules. Immediately, We obtain the following result by Theorem \ref{thm:main 1}.

\begin{corollary} \label{coro:quasitilted}
Let $A$ be a gentle algebra with global dimension $2$. If one of the conditions \ref{thm-main:(1)} and \ref{thm-main:(2)} given in Theorem \ref{thm:main 1} holds, then $A$ is quasi-tilted.
\end{corollary}

\begin{proof}
For any non-projective and non-injective indecomposable module $M$, we obtain
\begin{center}
$\pdim M\=< \gldim A = 2$ and $\idim M \=< \gldim A = 2$.
\end{center}
By Theorem \ref{thm:main 1}, we get $\idim M + \pdim M \=< 2\cdot \gldim A -1 = 3$.
Thus, $\pdim M = 2$ admits $\idim M = 1$, and $\pdim M = 1$ admits $\idim M = 2$ as required.
\end{proof}

Similarly, we have a result by Theorem \ref{thm:main 2}.

\begin{corollary} \label{coro of thm-main 2}
Let $A$ be a gentle algebra with 
$\gldim A=\infty$ and 
$\fdim A=2$.
If one of the conditions \ref{thm-main 2:(1)} and \ref{thm-main 2:(2)} given in Theorem \ref{thm:main 2} holds, then for each indecomposable module $M \in \ind(\modcat A)$ with $\pdim M + \idim M < \infty$, at least one of $\pdim M \=<1$ and $\idim M \=<1$ holds.
\end{corollary}

\begin{proof}
Its proof is parallel to that of Corollary \ref{coro:quasitilted}.
\end{proof}

\begin{remark} \rm
Corollary \ref{coro:quasitilted} is a method to judge a gentle algebra to be quasi-tilted,
and Corollary \ref{coro of thm-main 2} is parallel to Corollary \ref{coro:quasitilted}.
Moreover, it is well-known that a gentle algebra $A$ with finite finitistic dimension has a finite global dimension or an infinite global dimension, obviously, and if its global dimension $\gldim A$ is finite, then we have $\gldim A = \fdim A$.
Thus, naturally, Corollary \ref{coro of thm-main 2} admits Corollary \ref{coro:quasitilted},
since $\fdim A = 2$ follows that one of $\gldim A = 2$ and $\fdim A = 2 < \gldim A = \infty$ holds for $A$ to be gentle.
That is,
\begin{itemize}
  \item[] \hspace{0.6cm} {\it if the finitistic dimension $\fdim A$ of a gentle algebra $A$ is less than or equal to $2$, and the starting points of all maximal forbidden paths are strong sources or the ending points of all maximal forbidden paths are strong sinks,
then for each indecomposable module $M$ with $\pdim M + \idim M <+\infty$, we have
$\pdim M + \idim M < 2 \cdot \fdim A - 1$.}
\end{itemize}
The global dimension reflects the homology complexity of an algebra.
Then the fact, i.e., Corollary \ref{coro of thm-main 2} admitting Corollary \ref{coro:quasitilted}, provides a viewpoint that homology complexity can be described by finitistic dimension.
\end{remark}

\subsection{A sufficient and necessary condition} \label{subsect:suff and necess}

Coelho and Tosar provided a method to judge a gentle algebra to be quasi-tilted in \cite[Theorem 7]{CT2009} by using $\mathrm{hw}(\pmb{X}^{\bullet}):= \max\{t \in \ZZ\mid \HH^t(\pmb{X}^{\bullet})\ne 0\}
- \min\{t \in \ZZ\mid \HH^t(\pmb{X}^{\bullet})\ne 0\} + 1$ for each indecomposable object $\pmb{X}^{\bullet}$ in the derived category $D^b(A)$ of $A$.
In \cite[etc]{ZH2016, Zhang2019Indecomposables}, the authors call $\mathrm{hw}(\pmb{X}^{\bullet})$ the cohomological width of $\pmb{X}^{\bullet}$.

Let $\str$ be a string. We call it satisfies the {\defines quasi-tilted condition} if one of the following conditions holds.
\begin{enumerate}[label={\rm(Qt\arabic*)}]
  \item $\str$ is of the form
    $ \xymatrix{ \bullet \ar@{<~}[r]^{F_1} & \bullet \ar@{~}[r]^{\str} & \bullet \ar@{~>}[r]^{F_2} & \bullet } $
    in $(\Q,\I)$, where $F_1$ and $F_2$ are maximal forbidden paths of length $\=< 2$;
    \label{QT 1}
  \item $\str$ is of the form
    $ \xymatrix{ \bullet \ar@{~>}[r]^{F_1} & \bullet \ar@{~}[r]^{\str} & \bullet \ar@{<~}[r]^{F_2} & \bullet } $
    in $(\Q,\I)$, where $F_1$ and $F_2$ are maximal forbidden paths of length $\=< 2$;
    \label{QT 2}
  \item if there exists a maximal forbidden path $F_1$ of length two with $\t(F_1)=\s(\str)$,
    then for any forbidden path $F_2 = f_{21}f_{22}$ of length two,
    $\t(\str)\in \mathrm{v}(F_2)$ admits $\t(\str) = \t(f_{21})$,
    where $\mathrm{v}(F_2)$ is the set of all vertices of $F_2$;
    \label{QT 3}
  \item if there exists a maximal forbidden path $F_2$ of length two with $\s(F_2)=\t(\str)$,
    then for any forbidden path $F_1 = f_{11}f_{12}$ of length two,
    $\s(\str)\in \mathrm{v}(F_1)$ admits $\s(\str) = \t(f_{11})$,
    where $\mathrm{v}(F_1)$ is the set of all vertices of $F_1$.
    \label{QT 4}
\end{enumerate}

\begin{lemma} \label{lemm:QT}
Let $A=\kk\Q/\I$ be a quasi-tilted algebra. Then every string in $(\Q,\I)$ satisfies one of \ref{QT 1}, \ref{QT 2}, \ref{QT 3}, and \ref{QT 4}.
\end{lemma}

\begin{proof}
First of all, for a gentle pair $(\Q,\I)$, any string $\str$ in a gentle pair must be one of the forms (1) -- (4) shown in \Pic \ref{fig:QT str}, where in the case (1) or (2), $F_1$ and $F_2$ are unique, and in the case (3) or (4), $F_1'F_1''$ and $F_2'F_2''$.
\begin{figure}[htbp]
  \centering
\begin{tikzpicture}
\draw (0,0) node{$\xymatrix{ \bullet \ar@{<~}[r]^{F_1}
& v \ar@{~}[r]^{\str}
& w \ar@{~>}[r]^{F_2}
& \bullet }$};
\draw (0,-1) node{(1)};
\end{tikzpicture}
\ \
\begin{tikzpicture}
\draw (0,0) node{$\xymatrix{ \bullet \ar@{~>}[r]^{F_1}
& v \ar@{~}[r]^{\str}
& w \ar@{<~}[r]^{F_2}
& \bullet }$};
\draw (0,-1) node{(2)};
\end{tikzpicture}
\\
\begin{tikzpicture}
\draw (0,0) node{$\xymatrix{
\bullet \ar@{~>}[dd] \ar@{}[d]_{F_1'} && \bullet \ar@{~>}[dd] \ar@{}[d]^{F_2'}\\
 \bullet \ar@{~}[rr]^{\str} \ar@{}[d]_{F_1''}
&& \bullet \ar@{}[d]^{F_2''}  \\
\bullet && \bullet
}$};
\draw (0,-1.75) node{(3)};
\end{tikzpicture}
\ \ \ \ \
\begin{tikzpicture}
\draw (0,0) node{$\xymatrix{
\bullet \ar@{<~}[dd] \ar@{}[d]_{F_1''} && \bullet \ar@{~>}[dd] \ar@{}[d]^{F_2'}\\
 \bullet \ar@{~}[rr]^{\str} \ar@{}[d]_{F_1'}
&& \bullet \ar@{}[d]^{F_2''}  \\
\bullet && \bullet
}$};
\draw (0,-1.75) node{(4)};
\end{tikzpicture}
  \caption{\textsf{
  In (1) (resp., (2)), $F_1$ and $F_2$ are either maximal forbidden paths of length $>0$ or forbidden paths of length $0$, and $v$ and $w$ are vertices without relation on $F_1^{-1}\str F_2$ (resp., $F_1\str F_2^{-1}$); in (3) and (4), $F_1'$, $F_2'$, $F_1''$ and $F_2''$ are forbidden paths of length $\>=1$ such that $F_1'F_1''$ and $F_2'F_2''$ are maximal forbidden paths}}
  \label{fig:QT str}
\end{figure}
Next, assume that $A$ is quasi-tilted. Then, in the cases of (1) and (2) shown in \Pic \ref{fig:QT str},
the lengths of $F_1$ and $F_2$ must be $\=< 2$, since, by using Proposition \ref{prop:idim/pdim}, we have
\begin{align*}
 & \pdim \M(\str) = \max\{\ell(F_1),\ell(F_2)\} \=< \gldim A = 2 \\
~\Rightarrow~ & \pdim \M(\str) + \idim \M(\str) \=< 2+1=3
\end{align*}
and
\begin{align*}
  & \idim \M(\str) = \max\{\ell(F_1),\ell(F_2)\} \=< \gldim A = 2 \\
~\Rightarrow~ & \pdim \M(\str) + \idim \M(\str) \=< 1+2=3,
\end{align*}
respectively. It follows that $\str$ satisfies \ref{QT 1} if $\str$ is of the form shown in \Pic \ref{fig:QT str} (1),
or satisfies \ref{QT 2} if $\str$ is of the form shown in \Pic \ref{fig:QT str} (2).
One can check that if $\str$ is of the form shown in \Pic \ref{fig:QT str} (3) or (4),
then the lengths of $F_1'$, $F_2'$, $F_1''$, and $F_2''$ equal to $1$ by $\gldim A = 2$ and Proposition \ref{prop:idim/pdim}.
In this cases, $\str$ satisfies \ref{QT 3} or satisfies \ref{QT 4}.
\end{proof}

\begin{corollary} \label{coro:main 3}
Let $A=\kk\Q/\I$ be a gentle algebra with $\gldim A = 2$. Then $A$ is quasi-tilted if and only if
all strings in its bound quiver $(\Q,\I)$ satisfy one of the four conditions \ref{QT 1}--\ref{QT 4}.
\end{corollary}

\begin{proof}
Let $\str = a_1\cdots a_m$ be any string in $(\Q,\I)$.

If $\str$ satisfies \ref{QT 1}, then we have $\pdim\M(\str) = \max\{\ell(F_1), \ell(F_2)\} \=< 2$ by Proposition \ref{prop:idim/pdim}.
Next, we prove $\idim \M(\str) \=< 1$. To do this, we assume that $\idim \M(\str)\>= 2$, then at least one of the following two cases holds:
\begin{enumerate}[label={\rm(\arabic*)}]
  \item there is an arrow $\alpha$ which is not the first arrow $f_1$ of $F_1$ such that $\t(\alpha)=\s(F_1)~(=\s(\str))$;
    \label{main 3:case 1}
  \item there is an arrow $\beta$ which is not the first arrow $g_1$ of $F_2$ such that $\t(\beta)=\s(F_2)~(=\t(\str))$.
    \label{main 3:case 2}
\end{enumerate}

In \ref{main 3:case 1}, if $a_1\in\Q_1^{-1}$, then one of $\alpha f_1\in \I$ and $a_1^{-1}f_1\in\I$ holds since $A$ is a gentle algebra. It contradicts with $F_1$ is a maximal forbidden path. Thus, $a_1\in\Q_1$, then $\alpha a_1 \in \I$ by using $A$ to be gentle.
Similarly, in \ref{main 3:case 2}, we have $a_m\in\Q_1^{-1}$ and $\beta a_m^{-1} \in \I$.
Then we obtain that $\M(\str)$ is injective, it contradicts with $\idim \M(\str)\>= 2$.
The case of $\str$ satisfying \ref{QT 2} is similar.

If $\str$ satisfies \ref{QT 3}, and there exists a maximal forbidden path $F_1=f_{11}f_{12}$ ($f_{11}, f_{12}\in\Q_1$) of length two with $\t(F_1)=\s(\str)$,
then for any forbidden path $F_2$ with $\t(\str)\in \mathrm{v}(F_2)$,
where $\mathrm{v}(F_2)$ is the set of all vertices of $F$, we have $\ell(F_2) \=< 2$ by Theorem \ref{thm:LGH2024}.
If $F_2$ is maximal forbidden and $\ell(F_2)\=< 1$, then $\str$ satisfies \ref{QT 1}, We talked about this situation in the previous text.
If $\ell(F_2)=2$, then $F_2$ must be a maximal forbidden path which is of the form $F_2=f_{21}f_{22}$ ($f_{21},f_{22}\in\Q_1$) and $\t(\str) = \t(f_{21})$ holds.
In this case, we have $\idim \M(\str) = \ell(F_1) = 2$ by Proposition \ref{prop:idim/pdim}.
Next, we prove $\pdim \M(\str) \=< 1$. To do this, we assume $\pdim \M(\str) \>= 2$.
Since $A$ is gentle, one can check that there exists an arrow $\alpha$ with $\s(\alpha) = \s(\str)$ ($=\t(F_1)$) such that
$f_{12}\alpha \notin \I$, $a_1\in\Q_1^{-1}$, and $a_1^{-1}\alpha\in \I$ hold, see \Pic \ref{fig:thm-main 3}.
\begin{figure}[htbp]
  \centering
\begin{tikzpicture}
\draw[cyan!25][line width=10pt][shift={(0.1,-0.78)}] (-2.5,0)--(2.5,0) ;
\draw (0,0) node{
\xymatrix@C=1.25cm@R=1.25cm{
 \bullet \ar[d]^{f_{11}} \ar@{.}@/_1pc/[dd] & & &  \\
 \bullet \ar[d]^{f_{12}} & & & \bullet \ar[d]_{f_{21}} \ar@{.}@/^1pc/[dd] \\
 \bullet \ar@{}[rrr]^{\color{cyan}\str} \ar[d]_{\alpha} & \ar@{.}[ld] \ar@{~}[rr] \ar[l]^{a_1^{-1}} &  & \bullet \ar[d]_{f_{22}} \\
 \bullet  & & & \bullet
}
};
\end{tikzpicture}
  \caption{\textsf{$a_1^{-1}\alpha\in \I$, where $\str=a_1\cdots a_m$}}
  \label{fig:thm-main 3}
\end{figure}
Then we obtain $\pdim \M(\str) = \ell(F')$ by Proposition \ref{prop:idim/pdim} where $F'$ is a forbidden path with $\s(F')=\t(\str)$ whose length $\ell(F')$ is less than or equal to $2$. By using \ref{QT 3}, $\ell(F')$ must equal to $1$. Thus, $\pdim \M(\str) = 1$, a contradiction.
The case of $\str$ satisfying \ref{QT 4} is similar.

Therefore, we have $\pdim \M(\str)+\idim \M(\str) \=< 3$ for any string module $\M(\str)$.
It follows that $A$ is quasi-tilted if $\gldim A = 2$.
Moreover, Lemma \ref{lemm:QT} implies that the inverse of the above result holds; therefore, we complete this proof.
\end{proof}

\section{Examples} \label{sect:examp}

We provide some examples in this section.

\subsection{Examples for Theorems \ref{thm:main 1} and \ref{thm:main 2}}

Now we give an example to illustrate Theorems \ref{thm:main 1}.

\begin{example} \rm
Let $A = \kk\Q/\I$ be a gentle algebra given by $\Q = $
\[
\xymatrix{
  1' \ar[r]^{b}
& 1 \ar[r]^{a_1'} \ar[d]_{a_1}
& 2 \ar[d]^{a_2}
& 2' \ar[l]_{c} \\
& 3 \ar[r]_{a_3}
& 4
& }
\]
\[ \I = \langle ba_1, ca_2 \rangle. \]
Then the projective dimensions of all indecomposable modules are $1$ except $S(1')$ and $S(2')$.
Since $S(1')$ and $S(2')$ are injective, and $\gldim A = 2$ (by using Theorem \ref{thm:LGH2024}),
we have \[\pdim M + \idim M \=< 1+2 = 3 = 2\cdot \gldim A - 1\]
for any indecomposable module $M \notin \{S(1'), S(2')\}$ (up to isomorphism).
Moreover, one can check $\pdim S(1')=\pdim S(2')=2$, then we obtain $\pdim S(1')+\idim S(1') = \pdim S(2')+\idim S(2')=2 < 2\cdot \gldim A - 1$.
Thus, $A$ is a representation-infinite gentle algebra that conforms Theorem \ref{thm:main 1}.
\end{example}

Now we give another example to illustrate Theorem \ref{thm:main 2}.

\begin{example} \rm
Let $A=\kk\Q/\I$ be a gentle algebra given by $\Q=$
\begin{figure}[H]
  \centering
\begin{tikzpicture}[scale=0.5]
\foreach \x in {0,120,240}
\draw[->][rotate around = {10+\x:(0,0)}][line width=1pt] (2,0) arc(0:100:2);
\draw (2,0) node{$1$} (-1*1,1.73*1) node{$2$} (-1*1,-1.73*1) node{$3$};
\draw (4,0) node{$4$} (-1*2,1.73*2) node{$5$} (-1*2,-1.73*2) node{$6$};
\draw (6,0) node{$7$} (-1*3,1.73*3) node{$8$} (-1*3,-1.73*3) node{$9$};
\foreach \x in {0,120,240}
\draw[->][rotate around = {\x:(0,0)}][line width=1pt] (3.7,0) -- (2.3,0);
\foreach \x in {0,120,240}
\draw[->][rotate around = {\x:(0,0)}][line width=1pt] (5.7,0) -- (4.3,0);
\draw[rotate around = { 60:(0,0)}] (2.4,0) node{$a_{12}$};
\draw[rotate around = {180:(0,0)}] (2.4,0) node{$a_{23}$};
\draw[rotate around = {300:(0,0)}] (2.4,0) node{$a_{31}$};
\draw[rotate around = {  0:(0,0)}] (3,0.5) node{$a_{41}$};
\draw[rotate around = {120:(0,0)}] (3,0.5) node{$a_{52}$};
\draw[rotate around = {240:(0,0)}] (3,0.5) node{$a_{63}$};
\draw[rotate around = {  0:(0,0)}] (5,0.5) node{$a_{74}$};
\draw[rotate around = {120:(0,0)}] (5,0.5) node{$a_{85}$};
\draw[rotate around = {240:(0,0)}] (5,0.5) node{$a_{96}$};
\foreach \x in {0,120,240}
\draw[line width=1pt][dotted][red][rotate around = {\x:(0,0)}] (1.85,-0.5) arc(-90:-270:0.5);
\foreach \x in {0,120,240}
\draw[line width=1pt][dashed][violet][rotate around = {\x:(0,0)}] (3,0) arc(-180:0:1);
\end{tikzpicture}
\end{figure}
\noindent
and {$\I = \langle a_{12}a_{23}, a_{23}a_{31}, a_{31}a_{12},
  a_{74}a_{41}, a_{85}a_{52}, a_{96}a_{63} \rangle$}.
Then one can check that all indecomposable modules with finite projective dimension and injective dimension are listed as follows.
\begin{align*}
 & S(4)= \bsm 4 \esm
&& S(5)= \bsm 5 \esm
&& S(6)= \bsm 6 \esm \\
 & S(7)= \bsm 7 \esm = E(7)
&& S(8)= \bsm 8 \esm = E(8)
&& S(9)= \bsm 9 \esm = E(9) \\
 & \M(a_{41}a_{31}^{-1}) = \bsm 4&&3\\&1& \esm
&& \M(a_{63}a_{23}^{-1}) = \bsm 6&&2\\&3& \esm
&& \M(a_{52}a_{12}^{-1}) = \bsm 5&&1\\&2& \esm \\
 & P(1)= \bsm 1 \\2 \esm
&& P(2)= \bsm 2 \\3 \esm
&& P(3)= \bsm 3 \\1 \esm \\
 & P(4)= \bsm 4 \\1 \\2 \esm
&& P(5)= \bsm 5 \\2 \\3 \esm
&& P(6)= \bsm 6 \\3 \\1 \esm \\
 & P(7)= \bsm 7\\4 \esm = E(4)
&& P(8)= \bsm 8\\5 \esm = E(5)
&& P(9)= \bsm 9\\6 \esm = E(6) \\
 & E(1) = \bsm 6&&\\3&&4\\&1& \esm
&& E(2) = \bsm 4&&\\1&&5\\&2& \esm
&& E(3) = \bsm 5&&\\2&&6\\&3& \esm
\end{align*}
The projective resolution and injective resolution of $S(4)$ are
\[ 0 \To{} P(1) \To{} P(4) \To{} S(4) \To{} 0 \]
and
\[ 0 \To{} S(4) \To{} E(4) \To{} E(7) \To{} 0,\]
respectively. Then we have $\pdim S(4) = 1 = \idim S(4)$.
Similarly, we have $\pdim S(5) = 1 = \idim S(5)$ and $\pdim S(6) = 1 = \idim S(6)$.
The projective resolution and injective resolution of $\M(a_{41}a_{31}^{-1}) = \bsm 4&&3\\&1& \esm$ are
\[ 0 \To{} P(1) \To{} P(4)\oplus P(3) \To{} \M(a_{41}a_{31}^{-1}) \To{} 0  \]
and
\[ 0 \To{} \M(a_{41}a_{31}^{-1}) \To{} E(4) \To{} E(7) \To{} 0, \]
respectively. Then we have $\pdim \M(a_{41}a_{31}^{-1}) = 1 = \idim \M(a_{41}a_{31}^{-1})$.
Similarly, we have $\pdim \M(a_{63}a_{23}^{-1}) = 1 = \idim \M(a_{63}a_{23}^{-1})$
and $\pdim \M(a_{52}a_{12}^{-1}) = 1 = \idim \M(a_{52}a_{12}^{-1})$.
Notice that we have $\fdim A = \idim A = \pdim D(A) = 2$ in this example,
then we have \[\displaystyle 
\Hbound A
= 2 < 2\cdot \fdim A -1 = 3.\]
\end{example}

\subsection{An Example for Corollary \ref{coro:main 3}}

We give an example to illustrate Corollary \ref{coro:main 3}.

\begin{example} \rm
Let $A=\kk\Q/\I$ be a gentle algebra given by $\Q=$
\begin{figure}[H]
  \centering
\begin{tikzpicture}[scale=0.5]
\foreach \x in {0,120,240}
\draw[->][rotate around = {10+\x:(0,0)}][line width=1pt] (2,0) arc(0:100:2);
\foreach \x in {0,120,240}
\draw[->][rotate around = {5+\x:(0,0)}][line width=1pt] (6,0) arc(0:110:6);
\draw (2,0) node{$1$} (-1*1,1.73*1) node{$2$} (-1*1,-1.73*1) node{$3$};
\draw (4,0) node{$4$} (-1*2,1.73*2) node{$5$} (-1*2,-1.73*2) node{$6$};
\draw (6,0) node{$7$} (-1*3,1.73*3) node{$8$} (-1*3,-1.73*3) node{$9$};
\foreach \x in {0,120,240}
\draw[->][rotate around = {\x:(0,0)}][line width=1pt] (3.7,0) -- (2.3,0);
\foreach \x in {0,120,240}
\draw[->][rotate around = {\x:(0,0)}][line width=1pt] (5.7,0) -- (4.3,0);
\draw[rotate around = { 60:(0,0)}] (2.4,0) node{$a_{12}$};
\draw[rotate around = {180:(0,0)}] (2.4,0) node{$a_{23}$};
\draw[rotate around = {300:(0,0)}] (2.4,0) node{$a_{31}$};
\draw[rotate around = { 60:(0,0)}] (6.4,0) node{$a_{78}$};
\draw[rotate around = {180:(0,0)}] (6.4,0) node{$a_{89}$};
\draw[rotate around = {300:(0,0)}] (6.4,0) node{$a_{97}$};
\draw[rotate around = {  0:(0,0)}] (3,0.5) node{$a_{41}$};
\draw[rotate around = {120:(0,0)}] (3,0.5) node{$a_{52}$};
\draw[rotate around = {240:(0,0)}] (3,0.5) node{$a_{63}$};
\draw[rotate around = {  0:(0,0)}] (5,0.5) node{$a_{74}$};
\draw[rotate around = {120:(0,0)}] (5,0.5) node{$a_{85}$};
\draw[rotate around = {240:(0,0)}] (5,0.5) node{$a_{96}$};
\foreach \x in {0,120,240}
\draw[line width=1pt][dotted][red][rotate around = {\x:(0,0)}] (1.85,-0.5) arc(-90:-270:0.5);
\foreach \x in {0,120,240}
\draw[line width=1pt][dotted][blue][rotate around = {\x:(0,0)}] (5.95,-1) arc(-90:90:1);
\foreach \x in {0,120,240}
\draw[line width=1pt][dashed][violet][rotate around = {\x:(0,0)}] (3,0) arc(-180:0:1);
\end{tikzpicture}
\end{figure}
\noindent
and {$\I = \langle a_{12}a_{23}, a_{23}a_{31}, a_{31}a_{12},
  a_{78}a_{89}, a_{89}a_{97}, a_{97}a_{78},
  a_{74}a_{41}, a_{85}a_{52}, a_{96}a_{63} \rangle$}.
Then we have $\fdim A = 2$. We can find $39$ indecomposable modules (up to isomorphism) by using the Auslander--Reiten quiver of $A$, and then one can check that all indecomposable modules with finite projective dimension and injective dimension are listed as follows.
\begin{align*}
 & S(4)= \bsm 4 \esm
&& S(5)= \bsm 5 \esm
&& S(6)= \bsm 6 \esm \\
 & P(1) = \bsm 1\\2 \esm
&& P(2) = \bsm 2\\3 \esm
&& P(3) = \bsm 3\\1 \esm\\
 & P(4) = \bsm 4\\1\\2 \esm
&& P(5) = \bsm 5\\2\\3 \esm
&& P(6) = \bsm 6\\3\\1 \esm \\
 & P(7) = \bsm & 7 & \\ 4 & & 8 \\ & & 5 \esm
&& P(8) = \bsm & 8 & \\ 5 & & 9 \\ & & 6 \esm
&& P(9) = \bsm & 9 & \\ 6 & & 7 \\ & & 4 \esm \\
 & E(1) = \bsm 6&&\\3&&4\\&1& \esm
&& E(2) = \bsm 4&&\\1&&5\\&2& \esm
&& E(3) = \bsm 5&&\\2&&6\\&3& \esm\\
 & E(4) = \bsm 9\\7\\4 \esm
&& E(5) = \bsm 7\\8\\5 \esm
&& E(6) = \bsm 8\\9\\6 \esm\\
 & E(7) = \bsm 9\\7 \esm
&& E(8) = \bsm 7\\8 \esm
&& E(9) = \bsm 8\\9 \esm \\
 & \bsm 3&&4\\&1& \esm
&& \bsm 1&&5\\&2& \esm
&& \bsm 2&&6\\&3& \esm\\
 & \bsm &7&\\4&&8 \esm
&& \bsm &8&\\5&&9 \esm
&& \bsm &9&\\6&&7 \esm
\end{align*}
For any indecomposable projective (resp., injective) module $P$ (resp., $E$), we have $\idim P \=< 2$ (resp., $\pdim E \=< 2$) since $A$ is a Gorenstein algebra whose self-injective dimension is $2$.
For the simple module $S(4)$, we have that its projective resolution and injective resolution are
\[ 0 \To{} P(1) \To{} P(4) \To{} S(4) \To{} 0 \]
and
\[ 0 \To{} S(4) \To{} E(4) \To{} E(7) \To{} 0, \]
then $\pdim S(4)=1= \idim S(4)$. We have $\pdim S(5) = 1 = \idim S(5)$ and $\pdim S(6) = 1 = \idim S(6)$ by the symmetry of $(\Q,\I)$.
For the indecomposable $\bsm 3&&4\\&1& \esm$, we have its projective resolution and injective resolution are
\[ 0 \To{} P(1) \To{} P(3)\oplus P(4) \To{} \bsm 3&&4\\&1& \esm \To{} 0 \]
and
\[ 0 \To{} \bsm 3&&4\\&1& \esm \To{} E(1) \To{} E(6) \To{} E(9) \To{} 0, \]
then $\pdim \left( \bsm 3&&4\\&1& \esm \right)_A = 1$ and $\idim \left( \bsm 3&&4\\&1& \esm \right)_A = 2$,
i.e.,
\begin{center}
  $\pdim \left( \bsm 3&&4\\&1& \esm \right)_A + \idim \left( \bsm 3&&4\\&1& \esm \right)_A =1+2= 3$.
\end{center}
We have
\begin{center}
  $\pdim \left( \bsm 1&&5\\&2& \esm \right)_A + \idim \left( \bsm 1&&5\\&2& \esm \right)_A =1+2= 3$
\end{center}
\begin{center}
$\pdim \left( \bsm 2&&6\\&3& \esm \right)_A + \idim \left( \bsm 2&&6\\&3& \esm \right)_A =1+2= 3$
\end{center}
\begin{center}
$\pdim \left( \bsm &7&\\4&&8 \esm \right)_A + \idim \left( \bsm &7&\\4&&8 \esm \right)_A =2+1= 3$
\end{center}
\begin{center}
$\pdim \left( \bsm &8&\\5&&9 \esm \right)_A + \idim \left( \bsm &8&\\5&&9 \esm \right)_A =2+1= 3$
\end{center}
and
\begin{center}
$\pdim \left( \bsm &9&\\6&&7 \esm \right)_A + \idim \left( \bsm &9&\\6&&7 \esm \right)_A =2+1= 3$
\end{center}
by the symmetry of $(\Q,\I)$.
\end{example}


\subsection{Some counter examples}

The following two examples show that if a gentle pair $(\Q,\I)$ does not satisfy the conditions \ref{thm-main:(1)} and \ref{thm-main:(2)} given in Theorem \ref{thm:main 1}, then $\idim M + \pdim M = 2\cdot \gldim A$ may hold.
Here, the gentle algebra given in Example \ref{examp:1} (1) is representation-finite, i.e., the number of isoclasses of indecomposable modules is finite, and the gentle algebra given in Example \ref{examp:1} (2) is representation-infinite.

\begin{example} \rm \label{examp:1}
(1) Let $(\Q,\I)$ be a gentle pair with
\begin{center}
  $\Q = ~
  \xymatrix{1 \ar[r]^a & 2 \ar[r]^b & 3 \ar[r]^c & 4 \ar[r]^d & 5}$ and
  $\I = \langle ab, cd\rangle$.
\end{center}
Then $A=\kk\Q/\I$ is representation-finite.
The starting point $3$ of the maximal forbidden path $cd$ is not a strong source of $\Q$,
and we have $\pdim S(3)+\idim S(3)=2+2=4 > 2\cdot \gldim A-1 =3$. Here, one can check that $\gldim A=2$.

(2) Let $(\Q,\I)$ be a gentle pair with
\begin{center}
$\Q = \xymatrix{ 1 \ar[r]^{a} & 2 \ar@/^.5pc/[r]^{b_1} \ar@/_.5pc/[r]_{b_2} & 3 \ar[r]^c & 4 }$ and $\I = \langle ab_1, b_2c \rangle$.
\end{center}
Then $A=\kk\Q/\I$ is representation-infinite by using Theorem \ref{thm:WWBR}.
Consider the string module $M:=\M(b_2b_1^{-1}b_2) \cong ({_3}{^2}{_3}{^2})_A \in \ind(\modcat A)$, we have its projective resolution
\[ 0 -\!\!\!\to
P(4)
\To{}
P(3)^{\oplus 2} \oplus P(4)=
\left(
\begin{smallmatrix}
 3 \\ 4
\end{smallmatrix}
\right)^{\oplus 2}
\oplus
\begin{smallmatrix}
 4
\end{smallmatrix}
\To{}
P(2)^{\oplus 2} =
\left(
\begin{smallmatrix}
 & 2 & \\ 3 & & 3 \\ 4 & &
\end{smallmatrix}
\right)^{\oplus 2}
\To{} M
-\!\!\!\to 0 \]
and its injective resolution
\[
0 -\!\!\!\to M
\To{}
E(3)^{\oplus 2} =
\left(
\begin{smallmatrix}
 & & 1 \\ 2 & & 2 \\ & 3 &
\end{smallmatrix}
\right)^{\oplus 2}
\To{}
E(2)^{\oplus 2} \oplus E(1) =
\left(
\begin{smallmatrix}
 1 \\ 2
\end{smallmatrix} \oplus
\begin{smallmatrix}
 1
\end{smallmatrix}
\right)^{\oplus 2}
\To{}
E(1)
-\!\!\!\to 0.
\]
Then we have $\pdim M + \idim M = 4 > 2\cdot\gldim A - 1 = 3$, where $\gldim A = 2$ can be checked by computing all projective dimensions of simple modules or can be given by Theorem \ref{thm:LGH2024}.
\end{example}

However, there exist some examples for gentle algebras not satisfying the conditions \ref{thm-main:(1)} and \ref{thm-main:(2)} given in Theorem \ref{thm:main 1} such that $ 
 \Hbound A \=< 2\cdot \gldim A -1$ holds, see the following example.

\begin{example} \label{examp:3} \rm
(1) Let $(\Q,\I)$ be a gentle pair with
\begin{center}
  $\Q = ~
  \xymatrix{1 \ar[r]^a & 2 \ar[r]^b & 3 \ar[r]^c & 4 \ar[r]^d & 5}$ and
  $\I = \langle bc\rangle$.
\end{center}
Then $A$ is a representation-finite gentle algebra that does not satisfy two conditions \ref{thm-main:(1)} and \ref{thm-main:(2)} given in Theorem \ref{thm:main 1}.
We have $\gldim A = 2$, and for any indecomposable module $M$, we have $\pdim M + \idim M \=< 3 = 2\cdot\gldim A-1$.

(2) Let $A=\kk\Q/\I$ be a gentle algebra given by
\begin{center}
    $\Q = ~
  \xymatrix{1 \ar@/^0.6pc/[r]^{a_1} \ar@/_0.6pc/[r]_{a_2}
    & 2 \ar[r]^b & 3 \ar[r]^c & 4 \ar@/^0.6pc/[r]^{d_1} \ar@/_0.6pc/[r]_{d_2}  & 5}$ and
  $\I = \langle bc\rangle$,
\end{center}
then $A$ is a representation-infinite algebra does not satisfy two conditions \ref{thm-main:(1)} and \ref{thm-main:(2)} given in Theorem \ref{thm:main 1}. We have $\gldim A = 2$, and for any indecomposable module $M$, we have $\pdim M + \idim M \=< 3 = 2\cdot\gldim A-1$.
\end{example}

%

\section*{Authors' contributions}

\addcontentsline{toc}{section}{Authors' contributions}

All authors contributed equally to this work.

\section*{Acknowledgements}

\addcontentsline{toc}{section}{Acknowledgements}

\section*{Data Availability}

\addcontentsline{toc}{section}{Data Availability}

Data sharing is not applicable to this article as no datasets were generated or analysed
during the current study.


\addcontentsline{toc}{section}{References}

   \bibliographystyle{alpha} 

 \bibliography{referLiu20250515}

\end{document}